\documentclass[12pt]{conm-p-l}

\usepackage{hyperref}

\usepackage{amssymb}
\usepackage[capitalize]{cleveref}

\usepackage{accents} 
\usepackage{multicol}

\numberwithin{equation}{section}

\newtheorem{conj}[equation]{Conjecture}
\newtheorem{thm}[equation]{Theorem}

\newtheorem{cor}[equation]{Corollary}

\theoremstyle{definition}

\newtheorem{defn}[equation]{Definition}
\newtheorem*{defn*}{Definition}
\newtheorem{eg}[equation]{Example}
\newtheorem{egs}[equation]{Examples}
\newtheorem{rem}[equation]{Remark}

\newtheorem*{rem*}{Remark}
\newtheorem*{notation*}{Notation}
\newtheorem*{Dani}{\bf \MakeUppercase{Dani's contributions}}

\newcommand{\pref}[1]{{\rm(\ref{#1})}}

\newcommand{\bundle}{\mathcal{E}}
\newcommand{\real}{\mathbb{R}}
\newcommand{\integer}{\mathbb{Z}}
\newcommand{\rational}{\mathbb{Q}}
\newcommand{\complex}{\mathbb{C}}

\newcommand{\closure}[1]{\overline{#1}}
\newcommand{\Un}{\mathbb{U}_n}

\renewcommand{\vector}[1]{\accentset{\boldsymbol{\rightharpoonup}}{#1}}

\newcommand{\Rrank}{\mathop{\mathrm{rank}_{\real}}}
\newcommand{\Qrank}{\mathop{\mathrm{rank}_{\rational}}}

\DeclareMathOperator{\SL}{SL}
\DeclareMathOperator{\SO}{SO}
\DeclareMathOperator{\meas}{Meas}

\DeclareMathOperator{\Stab}{Stab}
\DeclareMathOperator{\Aut}{Aut}

\newcommand\bigset[2]{\left\{\, #1 
 \mathrel{\left| \vphantom {\left\{ #1 \mid #2 \right\} }
 \right.} #2 \,\right\} }

\makeatletter
\newcommand{\noprelistbreak}{\smallskip\@nobreaktrue\nopagebreak} 
\makeatother

\begin{document}

\title[Dani's Work on Homogeneous Dynamics]{Dani's Work on \\ Dynamical Systems on Homogeneous Spaces}

\author{Dave Witte Morris}

\address{Department of Mathematics and Computer Science \\ University of Lethbridge \\Lethbridge, Alberta, T1K~3M4 \\ Canada}

\email{Dave.Morris@uleth.ca}
\urladdr{http://people.uleth.ca/~dave.morris/}

\dedicatory{to Professor S.\,G.\,Dani on his 65th birthday}

\subjclass[2010]{Primary 
	37A17; 
	Secondary 
	11H55, 
	37A45
	} 

\date{\today} 

\begin{abstract}
We describe some of S.\,G.\,Dani's many contributions to the theory and applications of dynamical systems on homogeneous spaces, with emphasis on unipotent flows.
\end{abstract}

\maketitle

S.\,G.\,Dani has written over 100 papers. They explore a variety of topics, including:
\begin{multicols}{2} \raggedright
\noprelistbreak
	\begin{itemize} \itemsep=\smallskipamount
	\item flows on homogeneous spaces 
		\smallskip
		\noprelistbreak
		\begin{itemize} \itemsep=\smallskipamount 
		\item unipotent dynamics 
		\item applications to Number Theory 
		\item divergent orbits
		\item bounded orbits and Schmidt's Game 
		\item topological orbit equivalence
		\item Anosov diffeomorphisms
		\item entropy and other invariants
		\end{itemize}
	\columnbreak
	\item actions of locally compact groups
		\noprelistbreak
		\begin{itemize} \itemsep=\smallskipamount
		\item actions of lattices 
		\item action of $\Aut G$ on the Lie group~$G$
		\item stabilizers of points
		\end{itemize}
	\item convolution semigroups of probability measures on a Lie group
	\item finitely additive probability measures
	\item Borel Density Theorem
	\item history of Indian mathematics
	\end{itemize}
\end{multicols}
\noindent
Most of Dani's papers (about 60) are directly related to flows on homogeneous spaces. This survey will briefly discuss several of his important contributions in this field.

\begin{notation*}
Let:
\noprelistbreak
	\begin{itemize}
	\item $G = \SL(n,\real)$ (or, for the experts, $G$ may be any connected Lie group),
	\item $\{g^t\}$ be a one-parameter subgroup of~$G$,
	\item $\Gamma = \SL(n,\integer)$ (or, for the experts, $G$ may be any \emph{lattice} in~$G$),
	and
	\item $\varphi_t(x \Gamma) = g^t x \Gamma$ for $t \in \real$ and $x \Gamma \in G/\Gamma$.
	\end{itemize}
Then $\varphi_{s + t} = \varphi_{s} \circ \varphi_{t}$, so $\varphi_t$ is a flow on the homogeneous space $G/\Gamma$.
\end{notation*}

Much is known about flows generated by a general one-parameter subgroup $\{g^t\}$. (For example, the spectrum of $\{g^t\}$, as an operator on $L^2(G/\Gamma)$, is described in \cite[Thm.~2.1]{Dani-spectrum}. A few of Dani's other results on the general case are mentioned in \cref{OpenProbSect} below.) However, the most impressive results apply only to the flows generated by one-parameter subgroups that are ``unipotent:''

\begin{defn*}
A one-parameter subgroup~$\{u^t\}$ of $\SL(n,\real)$ is said to be \emph{unipotent} if $1$~is the only eigenvalue of~$u^t$ (for every $t \in \real$). This is equivalent to requiring that 
	$$ \text{$\{u^t\}$ is conjugate to a subgroup of
	$\Un =$
	\smaller$\begin{bmatrix}
	1  \\[-5pt]
	 & 1   \\[-5pt]
	&  & \ddots & \vbox to 0pt{\vss \hbox to 0pt{\hss\larger[5]$*$}\vskip15pt} \\[-5pt]
	&   & & 1 \\[-5pt]
	&  \vbox to 0pt{\vss \hbox to 0pt{\hskip0pt\larger[3]$0$\hss}\vskip0pt}  &  & & 1
	\end{bmatrix}$} .$$ 
\end{defn*}

Work of Dani, Margulis, Ratner, and others in the 1970's and 1980's showed that if $\{u^t\}$ is unipotent, then the corresponding flow on $G/\Gamma$ is surprisingly well behaved. 
In particular, the closure of every orbit is a very nice $C^\infty$ submanifold of $G/\Gamma$, and every invariant probability measure is quite obvious.

\Cref{UnipSect} describes these (and other) fundamental results on unipotent flows, and \Cref{ApplSect} explains some of their important applications in Number Theory. Dani's contributions are of lasting importance in both of these areas.

\section{Unipotent flows} \label{UnipSect}

\subsection{Orbit closure \pref{RatnerThm-orbit}, equidistribution \pref{RatnerThm-equi}, and measure classification  \pref{RatnerThm-measure}}
If $\{g^t\}$  is a (nontrivial) group of diagonal matrices, then there exists a $\{g^t\}$-orbit in $G/\Gamma$ that is very badly behaved --- its closure is a fractal (cf.\ \cite[Lem.~2]{Starkov-StructOrbs}).  Part~\pref{RatnerThm-orbit} of the following fundamental result tells us that unipotent flows have no such pathology. The other two parts describe additional ways in which the dynamical system is very nicely behaved.

\begin{thm}[Ratner \cite{Ratner-meas,Ratner-dist}] \label{RatnerThm}
If\/ $\{u^t\}$ is unipotent, then the following hold:
\noprelistbreak
	\begin{enumerate} \itemsep=\smallskipamount
	
	\renewcommand{\theenumi}{\textbf{O}}
	\item \label{RatnerThm-orbit}
	 The closure of every $u^t$-orbit on~$G/\Gamma$ is a\/ \textup{(}finite-volume, homogeneous\textup{)} $C^\infty$ submanifold.
			
	\renewcommand{\theenumi}{\textbf E}
	\item \label{RatnerThm-equi}
 	Every $u^t$-orbit on $G/\Gamma$ is uniformly distributed in its closure.

	\renewcommand{\theenumi}{\textbf  M}
	\item \label{RatnerThm-measure}
	Every ergodic $u^t$-invariant probability measure on~$G/\Gamma$ is the natural Lebesgue measure on some \textup{(}finite-volume, homogeneous\textup{)} $C^\infty$ submanifold.

	\end{enumerate}
\end{thm}

\begin{rem}
Here is a more precise statement of each part of \cref{RatnerThm}.

\pref{RatnerThm-orbit}
For each $a \in G/\Gamma$, there is a closed subgroup~$L = L_a$ of~$G$, such that $\closure{\{u^t\} a} = L a$. Furthermore, the closed submanifold $La$ of $G/\Gamma$ has finite volume, with respect to an $L$-invariant volume form on~$La$.

\pref{RatnerThm-equi}
Given $a \in G/\Gamma$, let $dg$~be the $L_a$-invariant volume form on~$\closure{\{u^t\} a}$ that is provided by~\pref{RatnerThm-orbit}. (After multiplying $dg$ by a scalar, we may assume the volume of $\closure{\{u^t\} a}$ is~$1$.) Then, for any continuous function~$f$ on $G/\Gamma$, with compact support, the average value of~$f$ along the orbit of~$a$ is equal to the average value of~$f$ on the closure of the orbit. That is, we have
	$$ \lim_{T \to \infty} \frac{1}{T} \int_0^T f (u^t a) \, dt = \int_{\closure{\{u^t\} a}} f \, dg .$$

\pref{RatnerThm-measure}
Suppose $\mu$ is a Borel measure on $G/\Gamma$, such that 
\noprelistbreak
	\begin{itemize}
	\item $\mu( G/ \Gamma) = 1$, 
	\item $\mu(u^t A) = \mu(A)$ for every measurable subset~$A$ of $G/\Gamma$,
	and
	\item for every $u^t$-invariant, measurable subset~$A$ of $G/\Gamma$, either $\mu(A) = 0$ or $\mu(A) = 1$.
	\end{itemize}
(In other words, $\mu$ is an ergodic $u^t$-invariant probability measure on $G/\Gamma$.)
Then there exists a closed subgroup $L = L_\mu$ of~$G$, and some $a = a_\mu \in G$, such that the orbit $L a$ is closed, and $\mu$~is an $L$-invariant Lebesgue measure on the submanifold~$L a$. (The $L$-invariant measure on $La$ is unique up to multiplication by a positive scalar, so the probability measure~$\mu$ is uniquely determined by~$L$.)
\end{rem}

\begin{Dani}
Dani was a central figure in the activity that paved the way for Ratner's proof of \cref{RatnerThm}.  In particular:
\noprelistbreak
	\begin{enumerate} \itemsep=\medskipamount
	
	\item Dani's early work on actions of unipotent subgroups was one of the ingredients that inspired Raghunathan to conjecture the truth of~\pref{RatnerThm-orbit}. Raghunathan did not publish the conjecture himself --- its first appearance in the literature was in a paper of Dani \cite[p.~358]{Dani-MinSetHoro}.
	
 	\item \pref{RatnerThm-measure} was conjectured by Dani \cite[p.~358]{Dani-MinSetHoro}. This was an important insight, because the methods developed in the 1980's were able to prove this part of \cref{RatnerThm} first, and the other two parts are corollaries of it\label{E-->O+M} (cf.\ \cref{LinearizationSect} below).

	\item For $G = \SL(2,\real)$, \pref{RatnerThm-measure} and \pref{RatnerThm-equi} were proved by Furstenberg \cite{Furstenberg-UniqErg} when $G/\Gamma$ is compact, but it was Dani who tackled the noncompact case. First, he \cite{Dani-InvtMeasNoncpct} proved \pref{RatnerThm-measure}. Then he \cite{Dani-UnifDist} proved the special case of \pref{RatnerThm-equi} in which $\Gamma = \SL(2,\integer)$. Subsequent joint work with Smillie \cite{DaniSmillie-UnifDist} established \pref{RatnerThm-equi} for the other (noncocompact) lattices in $\SL(2,\real)$.

	\item Dani \cite{Dani-MinSetHoro,Dani-OrbHoroFlows} proved analogues of \pref{RatnerThm-orbit} and \pref{RatnerThm-measure} in which $u^t$ replaced with the larger unipotent subgroup $\Un$\label{LargerAnalogue}. (See \cref{HoroSect} below for more discussion of this.)
	
	\item Dani and Margulis \cite{DaniMargulis-OrbClos} proved \pref{RatnerThm-orbit} under the assumption that $G = \SL(3,\real)$ and $u^t$ is generic.

	\end{enumerate}
\end{Dani}

Furthermore, Dani \cite{Dani-InvtMeasMargulis,Dani-OrbUnipFlow} established a very important special case of \pref{RatnerThm-equi}, long before the full theorem was proved. Note that if $a \in G/\Gamma$ and $\epsilon > 0$, then, since \pref{RatnerThm-orbit} tells us that $\closure{\{u^t\}a}$ has finite volume, there exists $C > 0$, such that the complement of the ball of radius~$C$ has volume less than~$\epsilon$. Therefore, letting 
\noprelistbreak
	\begin{itemize}
	\item $\lambda$ be the usual Lebesgue measure on~$\real$, 
	and
	\item $d(x,y)$ be the distance from $x$ to~$y$ in $G/\Gamma$,
	\end{itemize}
\pref{RatnerThm-equi} implies that 
	\begin{align} \label{DaniMargIneq} 
	 \limsup_{T \to \infty} \frac{\lambda \{\, t\in [0,T] \mid d(u^t a , a) > C \,\}}{T} < \epsilon 
	 . \end{align}
Dani proved this fundamental inequality: 

\begin{thm}[Dani \cite{Dani-InvtMeasMargulis,Dani-OrbUnipFlow}] \label{DaniMargNondiv}
For every $a \in G/\Gamma$ and $\epsilon > 0$, there exists $C > 0$, such that \pref{DaniMargIneq} holds.
\end{thm}

This is a strengthening of the following fundamental result:

\begin{cor}[Margulis \cite{Margulis-ActUnipLattSpaceSbornik,Margulis-ActUnipLattSpace}] \label{UnipNotDiverge}
No $u^t$-orbit diverges to $\infty$. 

	More precisely, for every $a \in G/\Gamma$, we have $d(u^t a, a) \not\to \infty$ as $t \to +\infty$.
\end{cor}

\begin{proof}
If $d(u^t a, a) \to \infty$, then there exists $T_0 > 0$, such that $d(u^t a , a) > C$ for all $t > T_0$. Therefore
	\begin{align*}
	\limsup_{T \to \infty} \frac{\lambda \{\, t\in [0,T] \mid d(u^t a , a) > C \,\}}{T}
	\ge \lim_{T \to \infty} \frac{T - T_0}{T}
	= 1
	> \epsilon
	.  & \qedhere \end{align*}
\end{proof}

The proof of \cref{DaniMargNondiv} employs an ingenious induction argument of Margulis that was appropriately modified by Dani. The best exposition of this idea is in an appendix of a paper by Dani and Margulis  \cite{DaniMargulis-ElemApproach}.

For applications, it is important to have strengthened versions of \cref{DaniMargNondiv} and \pref{RatnerThm-equi} giving estimates that are \emph{uniform} as the starting point of the orbit varies over a compact set. 
The first such theorems were proved by Dani and Margulis \cite{DaniMargulis-Asymp,DaniMargulis-LimitDist}.

\begin{rem} Margulis \cite[Rem.~3.12(II)]{Margulis-LieGrpsErgThy} observed that Dani's \cref{DaniMargNondiv} provides a short (and rather elementary) proof of the fundamental fact that arithmetic groups are lattices. That is, if $\mathbf{G}$ is an algebraic subgroup of $\mathbf{SL}_n$ that is defined over~$\rational$, and $\mathbf{G}$ has no characters that are defined over~$\rational$, then the homogeneous space $\mathbf{G}(\real)/\mathbf{G}(\integer)$ has finite volume.
\end{rem}

\subsection{Linearization} \label{LinearizationSect}
Dani and Margulis \cite[\S3]{DaniMargulis-LimitDist} developed an important method that is called ``\emph{Linearization},'' because it replaces the action of $u^t$ on $G/\Gamma$ with the much simpler action of $u^t$ by linear transformations on a vector space. It has become a crucial tool in applications of unipotent flows in Number Theory and related areas. 
To see the main idea (which has its roots in earlier work of Dani-Smillie \cite{DaniSmillie-UnifDist} and Shah \cite{Shah-UnifDistOrbs}), consider the following proof that \pref{RatnerThm-equi} is a consequence of \pref{RatnerThm-orbit} and~\pref{RatnerThm-measure}:

\begin{proof}[Idea of proof of \pref{RatnerThm-equi} from \pref{RatnerThm-orbit} and \pref{RatnerThm-measure}]
The start of the proof is straightforward.
Fix $a \in G/\Gamma$. From \pref{RatnerThm-orbit}, we know there is a closed subgroup~$L$ of~$G$, such that $\closure{\{u^t\} a} = L a$. Furthermore, there is an $L$-invariant probability measure~$\mu$ on $La$.

Assume, for simplicity, that $G/\Gamma$ is compact, and let $\meas(G/\Gamma)$ be the set of probability measures on $G/\Gamma$.
(The Riesz Representation Theorem tells us that $\meas(G/\Gamma)$ can be identified with the set of positive linear functionals on $C(G/\Gamma)$ that have norm~$1$, so it has a natural weak$^*$ topology.)
For each $T > 0$, define $M_T \in \meas(G/\Gamma)$ by
	$$ M_T(f) =  \frac{1}{T} \int_0^T f (u^t a) \, dt .$$
We wish to show $M_T \to \mu$ as $T \to \infty$.

Since $\meas(G/\Gamma)$ is a closed, convex subset of the unit ball in $C(G/\Gamma)^*$, the Banach-Alaoglu Theorem tells us that $\meas(G/\Gamma)$ is compact (in the weak$^*$ topology). Therefore, in order to show $M_T \to \mu$, it suffices to show that $\mu$~is the only accumulation point of $\{M_T\}$.

Thus, given an accumulation point $\mu_\infty$ of $\{M_T\}$, we wish to show $\mu_\infty = \mu$. It is not difficult to see that $\mu_\infty$ is $u^t$-invariant (since $M_T$ is nearly invariant when $T$ is large). Assume, for simplicity, that $\mu_\infty$ is ergodic. Then \pref{RatnerThm-measure} tells us there is a closed subgroup $L_\infty$ of~$G$, and some $a_\infty \in G$, such that the orbit $L_\infty a_\infty$ is closed, $\mu_\infty$~is $L_\infty$-invariant, and $\mu_\infty$~is supported on~$L_\infty a_\infty$. 

Since $M_T \to \mu_\infty$, we know $L_\infty a_\infty \subseteq \closure{\{u^t\} a} = L a$. To complete the proof, it suffices to show the opposite inclusion $L a \subseteq L_\infty a_\infty$. (This suffices because the two inclusions imply $L = L_\infty$, which means that $\mu_\infty$ must be the (unique) $L$-invariant probability measure on $La$, which is~$\mu$.)
 
We will now use \emph{Linearization} to show $u^t a \in L_\infty a_\infty$ for all~$t$. (Since $\closure{\{u^t\} a} = L a$, and $L_\infty a_\infty$ is closed, this implies the desired inclusion $L a \subseteq L_\infty a_\infty$.)
Roughly speaking, one can show that the subgroup~$L_\infty$ is Zariski closed (since the ergodicity of~$\mu_\infty$ implies that the unipotent elements generate a cocompact subgroup). Therefore, Chevalley's Theorem (from the theory of Algebraic Groups) tells us there exist:
\noprelistbreak
	\begin{itemize}
	\item a finite-dimensional real vector space~$V$, 
	\item a homomorphism $\rho \colon G \to \SL(V)$, 
	and
	\item $\vector v \in V$,
	\end{itemize}
such that $L_\infty = \Stab_G(\vector v)$. 
The Euclidean distance formula tells us that $d(\vector x, \vector v)^2$ is a polynomial of degree~$2$ (as a function of~$\vector x$). Also, since $\rho(u^t)$ is unipotent, elementary Lie theory shows that each matrix entry of $\rho(u^t)$ is a polynomial function of~$t$. Therefore, if we choose $g \in G$, such that $a = g \Gamma$, then
	$$d_g(t) := d \bigl( \rho(u^t g) \vector v, \vector v \bigr) $$
is a polynomial function of~$t$. Furthermore, the degree of this polynomial is bounded, independent of~$g$.

Since $M_T \to \mu_\infty$, and $\mu_\infty$ is supported on~$L_\infty a_\infty$, we know that if $T$~is large, then for most $t \in [0,T]$, the point $u^t a$ is close to $L_\infty a_\infty$. 
Assuming, for simplicity, that $a_\infty = e$, so $\rho(L_\infty a_\infty) \, \vector v = \rho(L_\infty) \, \vector v = \vector v$, 
this implies there exists $\gamma \in \Gamma$, such that $d_{g\gamma}(t)$ is very small. However, a polynomial of bounded degree that is very small on a large fraction of an interval must be small on the entire interval. We conclude that $d_{g\gamma}(t)$ is small for all $t \in \real^+$ (and that $\gamma$ is independent of~$t$). Since constants are the only bounded polynomials, this implies that the distance from $u^t a$ to $L_\infty a_\infty$ is a constant (independent of~$t$). Then, from the first sentence of this paragraph, we see that this distance must be~$0$, which means $u^t a \in L_\infty a_\infty$ for all~$t$, as desired.
\end{proof}

\begin{rem}
The above argument can be modified to derive both \pref{RatnerThm-orbit} and \pref{RatnerThm-equi} from~\pref{RatnerThm-measure}, without needing to assume~\pref{RatnerThm-orbit}. Therefore, as was mentioned on page~\pageref{E-->O+M}, \pref{RatnerThm-measure} implies both \pref{RatnerThm-orbit} and \pref{RatnerThm-equi}.
\end{rem}

\begin{rem}
The above proof assumes that $G/\Gamma$ is compact. To eliminate this hypothesis, one passes to the one-point compactification, replacing $\meas(G/\Gamma)$ with $\meas \bigl( G/\Gamma \cup \{\infty\} \bigr)$. The bulk of the argument can remain unchanged, because Dani's \cref{DaniMargNondiv} tells us that if $\mu_\infty$ is any accumulation point of $\{M_T\}$, then $\mu_\infty \bigl( \{\infty\} \bigr) = 0$, so $\mu_\infty \in \meas(G/\Gamma)$.
 \end{rem}

\subsection{Actions of horospherical subgroups and actions of lattices} \label{HoroSect}

As has already been mentioned on page~\pageref{LargerAnalogue}, Dani proved the analogues of \pref{RatnerThm-measure} and \pref{RatnerThm-orbit} in which the one-dimensional unipotent subgroup $\{u^t\}$ is replaced by a unipotent subgroup that is maximal:

\begin{thm}[Dani \cite{Dani-MinSetHoro,Dani-OrbHoroFlows}] \label{DaniMax}
Let
\noprelistbreak
	\begin{itemize}
	\item $G = \SL(n,\real)$ \textup(or, more generally, let $G$ be a connected, reductive, linear Lie group\textup),
	\item $U = \Un$ \textup(or, more generally, let $U$ be any maximal unipotent subgroup of~$G$\textup),
	and
	\item $\Gamma$ be a discrete subgroup of~$G$, such that $G/\Gamma$ has finite volume \textup(in other words, let $\Gamma$ be a lattice in~$G$\textup). 
	\end{itemize}
Then:
\noprelistbreak
	\begin{enumerate}

	\renewcommand{\theenumi}{$\mathbf{O}'$}
	\item \label{DaniMax-orbit}
	The closure of every $U$-orbit on~$G/\Gamma$ is a\/ \textup{(}finite-volume, homogeneous\textup{)} $C^\infty$ submanifold.

	\renewcommand{\theenumi}{$\mathbf{M}'$}
	\item \label{DaniMax-measure}
Every ergodic $U$-invariant probability measure on~$G/\Gamma$ is the natural Lebesgue measure on some \textup{(}finite-volume, homogeneous\textup{)} submanifold.
	\end{enumerate}
\end{thm}

\begin{rem} \label{HoroRem}
Although the statement of \cref{DaniMax} requires the unipotent subgroup~$U$ to be maximal,
Dani actually proved \pref{DaniMax-orbit} under the weaker assumption that $U$ is ``horospherical:''
	$$ \begin{matrix}
	\text{For each $g \in G$, the corresponding \emph{horospherical subgroup} is} \\
	 U_g = \{\, u \in G \mid \text{$g^{-k} u g^k \to e$ as $k \to +\infty$} \,\} 
	 . \end{matrix}$$
Since $1$~is the only eigenvalue of the identity matrix, and similar matrices have the same eigenvalues, it is easy to see that every element of~$U_g$ is unipotent. Conversely, the maximal unipotent subgroup~$\Un$ is horospherical. (Namely, we have $\Un = U_g$ if $g= \mathrm{diag}(\lambda_1,\lambda_2,\ldots,\lambda_n)$ is any diagonal matrix with $\lambda_1 > \lambda_2 > \cdots > \lambda_n > 0$.)
\end{rem}

Dani's first published paper was joint work with Mrs.~Dani \cite{DaniDani-DenseOrbits}, while they were students at the Tata Institute of Fundamental Research. It proved a $p$-adic version of the following interesting consequence of \pref{DaniMax-orbit}:

\begin{cor}[Greenberg \cite{Greenberg-DenseOrbits}] \label{GreenbergRnDense} 
Let $G = \SL(n,\real)$, and let\/ $\Gamma$ be a discrete subgroup of~$G$, such that $G/\Gamma$ is compact. Then the\/ $\Gamma$-orbit of every nonzero vector is dense in\/~$\real^n$.
\end{cor}

\begin{proof}
\Cref{DaniMax}\pref{DaniMax-orbit} tells us that, for each $g \in G$, there is a (unique) connected subgroup $L_g$ of~$G$, such that $\closure{\Un g \Gamma} = L_g \Gamma$. Since $\Un$ is horospherical (see \cref{HoroRem}) and $G/\Gamma$ is compact, one can show that $L_g = G$. This means $\Un g \Gamma$ is dense in~$G$ (for all $g \in G$).

So $\Gamma g \Un$ is dense in~$G$. 
Since $\Un$ fixes the vector $\vector{e_n} = (0,\ldots,0,1)$, and $G \vector {e_n} = \real^n \smallsetminus \{0\}$, this implies that $\Gamma g \vector{e_n}$ is dense in~$\real^n$. This is the desired conclusion, since $g \vector{e_n}$ is an arbitrary nonzero vector in~$\real^n$.
\end{proof}

\Cref{GreenbergRnDense} provides information about an action of the ``lattice''~$\Gamma$. This particular result is a consequence of the theory of unipotent dynamics, but other theorems of Dani about lattice actions do not come from this theory. As an example of this, we mention Dani's topological analogue of a famous measure-theoretic result of Margulis \cite[Thm.~1.14.2]{Margulis-QuotGrps}:

\begin{thm}[Dani \cite{Dani-EquiImage}] \label{QuotofG/P}
Let 
\noprelistbreak
	\begin{itemize}
	\item $G = \SL(n,\real)$ with $n \ge 3$, 
	\item $\Gamma = \SL(n,\integer)$,
	\item $B$ be the ``Borel'' subgroup of all upper-triangular matrices,
	\item $X$ be a compact, Hausdorff space on which~$\Gamma$ acts by homeomorphisms,
	and
	\item $\phi \colon G/B \to X$ be a continuous, surjective, $\Gamma$-equivariant map.
	\end{itemize}
Then $B$ is contained in a closed subgroup~$P$ of~$G$, such that $X$ is $\Gamma$-equivariantly homeomorphic to $G/P$.
\end{thm}

\begin{rem}
To make the statement more elementary, we assume $G =  \SL(n,\real)$ and $\Gamma = \SL(n,\integer)$ in \cref{QuotofG/P}, but Dani actually proved the natural generalization in which $G$ is allowed to be any semisimple Lie group of real rank greater than one, and $\Gamma$ is an irreducible lattice in~$G$.
\end{rem}

\section{Applications in Number Theory} \label{ApplSect}

\subsection{Values of quadratic forms at integer vectors}
Unipotent dynamics (which was discussed in \cref{UnipSect}) became well known through its use in the solution of the ``Oppenheim Conjecture.'' Before explaining this, let us look at a similar problem whose solution is much more elementary.

\begin{eg}
Suppose 
	$$L(x_1,x_2,\ldots,x_n) = \sum_i a_i x_i = a_1 x_1 + \cdots + a_n x_n$$ 
is a homogeneous polynomial of degree~$1$, with real coefficients. (Assume, to avoid degeneracies, that $L$~is not identically zero.) For any $b \in \real$, it is easy to find a solution of the equation $L(\vector x) = b$. 

However, a number theorist may wish to require the coordinates of~$\vector x$ to be integers (i.e., $\vector x \in \integer^n$).  Since $\integer^n$ is countable, most choices of~$b$ will yield an equation that does not have an integral solution, so it is natural to ask for an approximate solution: for every $\epsilon > 0$,
	$$ \text{does there exist $\vector x \in \integer^n$, such that $|L(\vector x) - b| < \epsilon$?} $$
Obviously, to say an approximate solution exists for every $b \in \real$ (and every $\epsilon > 0$) is the same as saying that $L(\integer^n)$ is dense in~$\real$.

It is obvious that if the coefficients of~$L$ are integers, then $L$~takes integer values at any point whose coordinates are integers. This means $L(\integer^n) \subseteq \integer$, so $L(\integer^n)$ is not dense in~$\real$. More generally, if $L$~is a scalar multiple of a polynomial with integer coefficients, then $L(\integer^n)$ is not dense in~$\real$. It is less obvious (but not difficult to prove) that the converse is true: 
	$$ \begin{matrix}
	\text{\it If $L(\vector x)$ is a a homogeneous polynomial of degree~$1$, with real coefficients,} \\
	\text{\it and $L$ is not a scalar multiple of a polynomial with integer coefficients,} \\
	\text{\it then $L(\integer^n)$ is dense in~$\real$}
	. \end{matrix} $$
\end{eg}

Everything in the above discussion is trivial, but the problem becomes extremely difficult if  polynomials of degree~$1$ are replaced by polynomials of degree~$2$. In this setting, a conjecture made by A.\,Oppenheim \cite{Oppenheim-Conj} in 1929 was not proved until almost 60 years later. The statement of the result needs to account for the following counterexamples:

\begin{egs}
Let $Q(x_1,x_2,\ldots,x_n) = \sum_{i,j} a_{i,j} x_i x_j$ be a homogeneous polynomial of degree~$2$, with real coefficients. (In other words, $Q$~is a ``real quadratic form.'')
\noprelistbreak
	\begin{enumerate} \itemsep=\smallskipamount
	
	\item We say $Q$ is \emph{positive-definite} if $Q(\real^n) \subseteq \real^{\ge0}$. (Similarly, $Q$ is \emph{negative-definite} if $Q(\real^n) \subseteq \real^{\le0}$.) Obviously, this implies that $Q(\real^n)$ is not dense in~$\real$, so it is obvious that the smaller set $Q(\integer^n)$ is not dense in~$\real$. For example, if we let $Q(x_1,x_2,\ldots,x_n) = a_1 x_1^2 + \cdots a_n x_n^2$, with each $a_i$~positive, then $Q$~is positive-definite, so $Q(\integer^n)$ is not dense in~$\real$.
	
\item Let $Q(x_1,x_2) = x_1^2 - \alpha^2 x_2^2$. It is not difficult to see that if $\alpha$ is badly approximable, then $0$ is \emph{not} an accumulation point of $Q(\integer^2)$. Obviously, then $Q(\integer^2)$ is not dense in~$\real$. (Recall that, to say $\alpha \in \real$ is \emph{badly approximable} means there exists $\epsilon > 0$, such that $|\alpha - (p/q)| > \epsilon/q^2$ for all $p,q \in \integer$. It is well known that quadratic irrationals, such as $1 + \sqrt{2}$, are always badly approximable.) This means that the ``obvious'' converse can fail when $n = 2$.

\item Given a $2$-variable counterexample $x_1^2 - \alpha^2 x_2^2$, it is easy to construct  counterexamples in any number of variables, such as
	$$Q(x_1,x_2,\ldots,x_n) = (x_1+\cdots+x_{n-1})^2 - \alpha^2 x_n^2 .$$
	Note that this quadratic form has $n$~variables, but a linear change of coordinates can transform it into a form with less than~$n$ variables. This means it is \emph{degenerate}.
	\end{enumerate}
\end{egs}

The following theorem shows that all quadratic counterexamples are of the above types.

\begin{thm}[Margulis \cite{Margulis-Formes}] \label{Margulis-OppenheimThm}
Suppose:
\noprelistbreak
	\begin{itemize} \itemsep=\smallskipamount
	\item $Q(x_1,x_2,\ldots,x_n)$ is a homogeneous polynomial of degree~$2$, with real coefficients,
	\item $Q$ is not a scalar multiple of a polynomial with integer coefficients,
	\item $Q$ is neither positive-definite nor negative-definite, 
	\item $n \ge 3$,
		and
	\item $Q$ is not degenerate.
	\end{itemize}
Then $Q(\integer^n)$ is dense in~$\real$.
\end{thm}


\begin{proof}[Idea of proof] 
We will explain how the result can be obtained as a corollary of Ratner's Theorem \pref{RatnerThm}. (The result was originally proved directly, because Ratner's theorem was not yet available in 1987.)

Assume, for simplicity, that $n = 3$. Given $b \in \real$, we wish to show there exists $\vector m \in \integer^3$, such that $Q(\vector m) \approx b$. Since $Q$ is neither positive-definite nor negative-definite, we have $Q(\real^3) = \real$, so there exists $\vector v \in \real^3$, such that $Q(\vector v) = b$. 

Let $G = \SL(3,\real)$, $\Gamma = \SL(3,\integer)$, and
	$$H = \SO_3(Q) = \{\, h \in \SL(3,\real) \mid \text{$Q(h \vector x) = Q(\vector x)$ for all $\vector x \in \real^3$} \,\} .$$
Note that $H$ is a subgroup of~$G$ that is generated (up to finite index) by unipotent one-parameter subgroups.
Therefore, Ratner's Theorem \pref{RatnerThm} implies that the closure $\closure{H \Gamma}$ of $H\Gamma$ is a very nice submanifold of~$G$. More precisely, there is a closed subgroup~$L$ of~$G$, such that $\closure{H \Gamma} = L \Gamma$, and $H \subseteq L$. However, it can be shown that $H$ is a maximal subgroup of~$G$, and, since $Q(\vector x)$ does not have integer coefficients (up to a scalar multiple), that $\closure{H \Gamma} \neq H \Gamma$ (so $L \neq H$). This implies $L = G$. In other words, $H \Gamma$ is dense in~$G$.

For convenience, let $\vector{e_1} = (1,0,0) \in \real^3$.
Since $G = \SL(3,\real)$ is transitive on the nonzero vectors in $\real^3$, we know there exists $g \in G$, such that $g \vector{e_1} = \vector v$. Then the conclusion of the preceding paragraph implies there exist $h \in H$ and $\gamma \in \Gamma = \SL(3,\integer)$, such that
	\begin{align} \tag{$*$} \label{hgammae1}
	h \gamma \vector {e_1} \approx \vector v 
	. \end{align}
Let $\vector m =  \gamma \vector{e_1} \in \integer^3$. Then 
	\begin{align*}
	Q(\vector m) &= Q (\gamma \vector{e_1}) 
	&& \text{(definition of~$\vector m$)} \\
	&= Q(h \gamma \vector{e_1}) 
	&& \text{(definition of~$H$)} \\
	&\approx Q(\vector v) 
	&& \text{(polynomial~$Q$ is continuous, and (\ref{hgammae1}))} \\
	&= b
	&& \text{(definition of~$\vector v$)} 
	. \qedhere \end{align*}
\end{proof}

\begin{Dani}
Joint work of Dani and Margulis made the following important improvements to this theorem:
\noprelistbreak
	\begin{enumerate} \itemsep=\smallskipamount
	
	\item[\cite{DaniMargulis-PrimIntPts}] \emph{approximation by primitive vectors:} $Q(\mathcal{P})$ is dense in~$\real$, where
	$$ \mathcal{P} = \{\, (m_1,m_2,\ldots,m_n) \in \integer^n \mid \gcd(m_1,m_2,\ldots,m_n) = 1 \,\} .$$
	
	\item[\cite{DaniMargulis-OrbClos}] \emph{simultaneous approximation:} Given two quadratic forms $Q_1$ and~$Q_2$ (satisfying appropriate conditions), the set 
		$ \bigset{ \bigl( Q_1(\vector m) , Q_2(\vector m) \bigr) }{ \vector m \in \mathcal{P} } $
is dense in~$\real^2$. 

	\item[\cite{DaniMargulis-LimitDist}] \emph{quantitative estimates:} For any nonempty open interval $I \subset \real$, \cref{Margulis-OppenheimThm} shows there exists at least one $\vector m \in \integer^n$ with $Q(\vector m) \in I$. In fact, there exist \emph{many} such~$\vector m$ (of bounded norm). Namely, if $\lambda$ is the Lebesgue measure on~$\real^n$, then
		$$ \liminf_{C \to \infty} \frac{
		 \# \bigset{ \vector m \in \integer^n }{ 
		 	\begin{matrix} Q(\vector m) \in I , \\[\smallskipamount] \| \vector m \| < C \end{matrix} }
		 }{
		\lambda \left( \bigset{ \vector v \in \real^n }{
			 \begin{matrix} Q(\vector v) \in I , \\[\smallskipamount] \| \vector v \| < C \end{matrix} } \right) 
		} \ge 1 .$$
	Furthermore, the estimate is uniform when $Q$ varies over any compact set of quadratic forms that all satisfy the hypotheses of \cref{Margulis-OppenheimThm}.

	\end{enumerate}
The paper \cite{DaniMargulis-LimitDist} has been especially influential, because it introduced \emph{Linearization} (which was discussed in \cref{LinearizationSect}). 
\end{Dani}

Dani has continued his contributions to this field of research in recent years. For example, he \cite{Dani-SimApproxQuadLin} proved a simultaneous approximation theorem for certain pairs consisting of a linear form and a quadratic form.

\section{Dynamics of general homogeneous flows} \label{OpenProbSect}

This section presents a few fundamental questions that Dani worked on, but remain open.

\subsection{Kolmogorov automorphisms}

\begin{defn}
If $p \colon X \to \real^+$ is a probability distribution on a finite set $X = \{x_1,\ldots,x_n\}$, then there is a natural product measure $p^\integer$ on the infinite product 
	$X^\integer = \{\, f \colon \integer \to X \,\} $,
and the associated \emph{Bernoulli shift} is the measurable map
	$ B_{X,p} \colon X^\integer \to X^\integer$,
defined by
	$ B_{X,p}(f)(k) = f(k-1)$.
\end{defn}

It is believed that almost all translations on homogeneous spaces are isomorphic to Bernoulli shifts:

\begin{conj} 
Suppose 
\noprelistbreak
	\begin{itemize}
	\item $G = \SL(n,\real)$ \textup(or, more generally, let $G$ be a connected, semisimple, linear Lie group with no compact factors\textup),
	\item $\Gamma = \SL(n,\integer)$ \textup(or, more generally, let $\Gamma$ be an irreducible lattice in~$G$\textup),
	\item $g \in G$, 
	and
	\item $T_g \colon G/\Gamma \to G/\Gamma$ be defined by $T_g(x \Gamma) = gx\Gamma$.
	\end{itemize}
If there is an eigenvalue~$\lambda$ of~$g$, such that $|\lambda| \neq 1$ \textup(in other words, if the entropy of~$T_g$ is nonzero\textup), then $T_g$ is measurably isomorphic to a Bernoulli shift.
\end{conj}

\begin{Dani} Two results of Dani \cite{Dani-Kolmogorov} are the inspiration for this conjecture:
\noprelistbreak
	\begin{enumerate}
	\item The conjecture is true if the matrix~$g$ is diagonalizable over~$\complex$.
	\item \label{DaniKThm}
	In the general case, $T_g$ has no zero-entropy quotients (if $\exists \lambda, |\lambda| \neq 1$).
	\end{enumerate}
A transformation with no zero-entropy quotients is called a ``\emph{Kolmogorov automorphism}.'' Examples of non-Bernoulli Kolmogorov automorphisms are rare, so \pref{DaniKThm} is good evidence that the conjecture is true for all~$g$, not just those that are diagonalizable. 
\end{Dani}

\subsection{Anosov diffeomorphisms}

\begin{defn}
A diffeomorphism~$f$ of a compact, connected manifold~$M$ is \emph{Anosov} if, at every point $x \in M$, the tangent space $T_x M$ has a splitting $T_x M = \bundle^+ \oplus \bundle^-$, such that	
\noprelistbreak
	\begin{itemize}
	\item for $v \in \bundle^+$, $D(f^{k})(v) \to 0$ exponentially fast as $k \to -\infty$,
	and
	\item for $v \in \bundle^-$, $D(f^{k})(v) \to 0$ exponentially fast as $k \to +\infty$.
	\end{itemize}
\end{defn}

\begin{conj}[from the 1960's] \label{AnosovConj}
If there is an Anosov diffeomorphism~$f$ on~$M$, then some finite cover of~$M$ is a nilmanifold. \textup(This means the cover is a homogeneous space $G/\Gamma$, where $G$ is a nilpotent Lie group, and\/ $\Gamma$~is a discrete subgroup of~$G$.\textup)
Furthermore, lifting $f$ to the finite cover yields an affine map on the nilmanifold.
\end{conj}

\begin{Dani}
\ 
\noprelistbreak
	\begin{enumerate}
	\item[\cite{Dani-NilAnosov}] Dani constructed many nilmanifolds that have Anosov diffeomorphisms. Much more recently, joint work with M.\,Mainkar \cite{DaniMainkar-Anosov} constructed examples of every sufficiently large dimension.
	\item[\cite{Dani-AffineAutsFP}] Dani proved \cref{AnosovConj} under the assumption that $M$ is  a double-coset space $K \backslash G / H$ (and certain additional technical conditions are satisfied).
	\end{enumerate}
\end{Dani}

\subsection{Divergent trajectories}
We know, from \cref{UnipNotDiverge}, that if $u^t$ is unipotent, then no $u^t$-orbit diverges to~$\infty$. On the other hand, orbits of a diagonal subgroup can diverge to~$\infty$.

\begin{eg}
Let
	$$G = \SL(2,\real),
	\quad
	a^t = \begin{bmatrix} e^{-t} & 0 \\ 0 & e^t \end{bmatrix},
	\quad 
	\Gamma = \SL(2,\integer),
	\text{\quad and\quad}
	\vector{e_1} = \begin{bmatrix} 1 \\ 0 \end{bmatrix} .$$
Then $a^t \, \vector{e_1} \to \vector 0$ as $t \to \infty$, so $a^t \, \Gamma \to \infty$ in $G/\Gamma$.
\end{eg}

The divergent orbits of the diagonal matrices in $\SL(2,\real) / \SL(2,\integer)$ are well known, and quite easy to describe. Dani vastly generalized this, by proving that all divergent orbits are obvious in a much wider setting. Here is a special case of his result:

\begin{thm}[Dani \cite{Dani-Divergent}] \label{DaniDivRank1}
Let 
\noprelistbreak
	\begin{itemize}
	\item $G = \SO(1,n)$ \textup(or, more generally, let $G$ be a connected, almost simple algebraic\/ $\rational$-subgroup of\/ $\SL_{n+1}(\real)$, with $\Qrank G = 1$\textup),
	\item $\Gamma = G \cap \SL_{n+1}(\integer)$,
	\item $\{a^t\}$ be a one-parameter subgroup of~$G$ that is diagonalizable over~$\real$,
	and
	\item $g \in G$.
	\end{itemize}
Then $a^t g \Gamma$ diverges to~$\infty$ in $G/\Gamma$ if and only if there exist
\noprelistbreak
	\begin{itemize}
	\item a continuous homomorphism $\rho \colon G \to \SL(\ell,\real)$, for some~$\ell$, such that $\rho(\Gamma) \subseteq \SL(\ell,\integer)$,
	and
	\item a nonzero vector $\vector v \in \integer^\ell$,
	\end{itemize}
such that $\rho(a^t g) \, \vector v \to \vector 0$ as $t \to \infty$.
\end{thm}

Surprisingly, he was also able to exhibit many cases in which there are divergent orbits that do not come from the construction in \cref{DaniDivRank1}:

\begin{thm}[Dani \cite{Dani-Divergent}] \label{DaniDivRank>1}
Let 
\noprelistbreak
	\begin{itemize}
	\item $G = \SL(n,\real)$, with $n \ge 3$ \textup(or, more generally, let $G$ be the\/ $\real$-points of a connected, almost simple algebraic\/ $\rational$-group with $\Qrank G = \Rrank G \ge 2$\textup),
	\item $\Gamma = \SL(n,\integer)$ \textup(or, in general, let $G$ be the\/ $\integer$-points of~$G$\textup),
	and
	\item $\{a^t\}$ be a\/ \textup(nontrivial\/\textup) one-parameter subgroup of~$G$ that is diagonalizable over\/~$\real$.
	\end{itemize}
Then there are $a^t$-orbits in $G/\Gamma$ that diverge to~$\infty$, but do not correspond to a continuous homomorphism $\rho \colon G \to \SL(\ell,\real)$, as in \cref{DaniDivRank1}.
\end{thm}

It remains an open problem to determine the set of divergent orbits in $G/\Gamma$ in cases where \cref{DaniDivRank1} does not apply.

\bibliographystyle{amsplain}

\begin{thebibliography}{DM90b..}

\larger 
\itemsep=\smallskipamount 

\bibitem[DD73]{DaniDani-DenseOrbits}
J.\,S.\,Dani and S.\,G.\,Dani:
Discrete groups with dense orbits,
\emph{J. Indian Math. Soc. (N.S.)} 37 (1973), 183--195. 
\MR{0360932}

\bibitem[D76]{Dani-Kolmogorov}
S.\,G.\,Dani:
Kolmogorov automorphisms on homogeneous spaces,
\emph{Amer. J. Math.} 98 (1976), no.~1, 119--163. 
\MR{0419728}

\bibitem[D77]{Dani-spectrum}
S.\,G.\,Dani:
Spectrum of an affine transformation,
\emph{Duke Math. J.} 44 (1977), no.~1, 129--155. 
\MR{0444835}

\bibitem[D78a]{Dani-InvtMeasNoncpct}
S.\,G.\,Dani:
Invariant measures of horospherical flows on noncompact homogeneous spaces,
\emph{Invent. Math.} 47 (1978), no.~2, 101--138. 
\MR{0578655}

\bibitem[D78b]{Dani-NilAnosov}
S.\,G.\,Dani:
Nilmanifolds with Anosov automorphism.
\emph{J. London Math. Soc.} (2) 18 (1978), no.~3, 553--559. 
\MR{0518242}

\bibitem[D79]{Dani-InvtMeasMargulis}
S.\,G.\,Dani:
On invariant measures, minimal sets and a lemma of Margulis,
\emph{Invent. Math.} 51 (1979), no.~3, 239--260. 
\MR{0530631}

\bibitem[D80]{Dani-AffineAutsFP}
S.\,G.\,Dani:
On affine automorphisms with a hyperbolic fixed point.
\emph{Topology} 19 (1980), no.~4, 351--365. 
\MR{0584560}

\bibitem[D81]{Dani-MinSetHoro}
S.\,G.\,Dani:
Invariant measures and minimal sets of horospherical flows,
\emph{Invent. Math.} 64 (1981), no.~2, 357--385.
\MR{0629475}

\bibitem[D82]{Dani-UnifDist}
S.\,G.\,Dani:
On uniformly distributed orbits of certain horocycle flows,
\emph{Ergodic Theory Dynamical Systems} 2 (1982), no.~2, 139--158.
\MR{0693971}

\bibitem[D84a]{Dani-OrbUnipFlow}
S.\,G.\,Dani:
On orbits of unipotent flows on homogeneous spaces,
\emph{Ergodic Theory Dynam. Systems} 4 (1984), no.~1, 25--34. 
\MR{0758891}

\bibitem[D84b]{Dani-EquiImage}
S.\,G.\,Dani:
Continuous equivariant images of lattice-actions on boundaries,
\emph{Ann. of Math.} (2) 119 (1984), no.~1, 111--119. 
\MR{0736562}

\bibitem[D85]{Dani-Divergent}
S.\,G.\,Dani:
Divergent trajectories of flows on homogeneous spaces and Diophantine approximation,
\emph{J. Reine Angew. Math.} 359 (1985), 55--89. 
\MR{0794799}

\bibitem[D86]{Dani-OrbHoroFlows}
S.\,G.\,Dani:
Orbits of horospherical flows,
\emph{Duke Math. J.} 53 (1986), no.~1, 177--188. 
\MR{0835804}

\bibitem[D08]{Dani-SimApproxQuadLin}
S.\,G.\,Dani:
Simultaneous Diophantine approximation with quadratic and linear forms,
\emph{J. Mod. Dyn.} 2 (2008), no.~1, 129--138.
\MR{2366232}


\bibitem[DMa05]{DaniMainkar-Anosov}
S.\,G.\,Dani and M.\,Mainkar:
Anosov automorphisms on compact nilmanifolds associated with graphs,
\emph{Trans. Amer. Math. Soc.} 357 (2005), no.~6, 2235--2251. 
\MR{2140439}

\bibitem[DM89]{DaniMargulis-PrimIntPts}
S.\,G.\,Dani and G.\,A.\,Margulis:
Values of quadratic forms at primitive integral points,
\emph{Invent. Math.} 98 (1989), no.~2, 405--424.
\MR{1016271}

\bibitem[DM90a]{DaniMargulis-OrbClos}
S.\,G.\,Dani and G.\,A.\,Margulis:
Orbit closures of generic unipotent flows on homogeneous spaces of $\SL(3,\real)$,
\emph{Math. Ann.} 286 (1990), no.~1--3, 101--128. 
\MR{1032925}

\bibitem[DM90b]{DaniMargulis-ElemApproach}
S.\,G.\,Dani and G.\,A.\,Margulis:
Values of quadratic forms at integral points: an elementary approach,
\emph{Enseign. Math.} (2) 36 (1990), no.~1--2, 143--174. 
\MR{1071418}

\bibitem[DM91]{DaniMargulis-Asymp}
S.\,G.\,Dani and G.\,A.\,Margulis:
Asymptotic behaviour of trajectories of unipotent flows on homogeneous spaces,
\emph{Proc. Indian Acad. Sci. Math. Sci.} 101 (1991), no.~1, 1--17. 
\MR{1101994}

\bibitem[DM93]{DaniMargulis-LimitDist}
S.\,G.\,Dani and G.\,A.\,Margulis:
Limit distributions of orbits of unipotent flows and values of quadratic forms,
in: S.\,Gel'fand and S.\,Gindikin, eds.,
\emph{I.\,M.\,Gel'fand Seminar.}
Amer. Math. Soc., Providence, RI, 1993,
pp.~91--137.
\MR{1237827}

\bibitem[DS84]{DaniSmillie-UnifDist}
S.\,G.\,Dani and J.\,Smillie:
Uniform distribution of horocycle orbits for Fuchsian groups,
\emph{Duke Math. J.} 51 (1984), no.~1, 185--194. 
\MR{0744294}

\bibitem[F73]{Furstenberg-UniqErg}
H.\,Furstenberg:
The unique ergodicity of the horocycle flow,
in A.\,Beck, ed.:
\emph{Recent Advances in Topological Dynamics.}
Springer, Berlin, 1973, pp. 95--115.
\MR{0393339}

\bibitem[G63]{Greenberg-DenseOrbits}
L.\,Greenberg, 
Discrete groups with dense orbits,
in 
L.\,Auslander, L.\,Green, and F.\,Hahn,
\emph{Flows on Homogeneous Spaces.}
Princeton University Press, Princeton, N.\,J., 1963, pp.~85--103.
\MR{0167569}

\bibitem[M71]{Margulis-ActUnipLattSpaceSbornik}
G.\,A.\,Margulis:
On the action of unipotent groups in a lattice space (\emph{Russian}),
\emph{Mat. Sb. (N.S.)} 86 (128) (1971), 552--556.
English translation in:
\emph{Math. USSR Sbornik} 15 (1971), no.~4, 549--554.
\MR{0291352}

\bibitem[M75]{Margulis-ActUnipLattSpace}
G.\,A.\,Margulis:
On the action of unipotent groups in the space of lattices,
in I.\,M.\,Gel'fand, ed.,
\emph{Lie Groups and Their Representations.}
Halsted (Wiley), New York, 1975, pp. 365--370.
\MR{0470140}

\bibitem[M78]{Margulis-QuotGrps}
G.\,A.\,Margulis:
Quotient groups of discrete subgroups and measure theory (\emph{Russian}),
\emph{Funk. Anal. Prilozhen} 12 (1978), no.~4, 64--76.
English translation in:
\emph{Func. Anal. Appl.} 12 (1978), no.~4, 295--305.
\MR{0545365}

\bibitem[M87]{Margulis-Formes}
G.\,A.\,Margulis:
Formes quadratriques ind\'efinies et flots unipotents sur les espaces homog\`enes,
\emph{C. R. Acad. Sci. Paris S\'er.~I Math.} 304 (1987), no.~10, 249--253. 
\MR{0882782}

\bibitem[M88]{Margulis-LieGrpsErgThy}
G.\,A.\,Margulis:
Lie groups and ergodic theory,
in
L.\,L.\,Avramov and K.\,B.\,Tchakerian, eds.,
 \emph{Algebra---Some Current Trends (Varna, 1986).}
Springer, New York, 1988, pp.~130--146. 
\MR{0981823}

\bibitem[O29]{Oppenheim-Conj}
A.\,Oppenheim:
The minima of indefinite quaternary quadratic forms,
\emph{Proc. Nat. Acad. Sci. U.S.A.} 15 (1929) 724--727.
 zbMATH: JFM~55.0722.01

\bibitem[R91a]{Ratner-meas}
M.\,Ratner:
On Raghunathan's measure conjecture,
\emph{Ann. of Math.} (2) 134 (1991), no.~3, 545--607. 
\MR{1135878}

\bibitem[R91b]{Ratner-dist}
M.\,Ratner:
Raghunathan's topological conjecture and distributions of unipotent flows,
\emph{Duke Math. J.} 63 (1991), no.~1, 235--280. 
\MR{1106945}

\bibitem[Sh91]{Shah-UnifDistOrbs}
N.\,A.\,Shah:
Uniformly distributed orbits of certain flows on homogeneous spaces,
\emph{Math. Ann.} 289 (1991), no.~2, 315--334.
\MR{1092178}

\bibitem[St90]{Starkov-StructOrbs}
A.\,N.\,Starkov:
The structure of orbits of homogeneous flows and the Raghunathan conjecture (\emph{Russian}),
\emph{Uspekhi Mat. Nauk} 45 (1990), no.~2 (272), 219--220.
English translation in:
\emph{Russian Math. Surveys} 45 (1990), no.~2, 227--228.
\MR{1069361}

\end{thebibliography}

\end{document}